\documentclass[12 pt]{amsart}
\usepackage{amssymb, amsmath, amsthm}
\usepackage{enumitem}
\usepackage{hyperref}
\usepackage{bbm}
\usepackage{float}
\usepackage{caption}
\usepackage{tikz-cd}
\usepackage{bbm}

\DeclareMathOperator{\Spec}{Spec}
\DeclareMathOperator{\CH}{CH}
\DeclareMathOperator{\fieldchar}{char}
\DeclareMathOperator{\id}{id}
\DeclareMathOperator{\Ann}{Ann}

\newcommand{\mg}[1]{\mathcal{M}_{#1}}
\newcommand{\mgbar}[1]{\overline{\mathcal{M}}_{#1}}

\newcommand{\mgrtilde}[2]{\widetilde{\mathcal{M}}_{#1}^{#2}}
\newcommand{\cgbar}[1]{\overline{\mathcal{C}}_{#1}}
\newcommand{\inv}{^{-1}}

\newtheorem{theorem}{Theorem}[section]
\newtheorem{lemma}[theorem]{Lemma}
\newtheorem{proposition}[theorem]{Proposition}

\newtheorem*{theorem*}{Theorem}

\theoremstyle{definition}
\newtheorem{definition}[theorem]{Definition}
\newtheorem{note}[theorem]{Note}

\title{The integral Chow ring of $\mgbar{1,3}$}
\author{Martin Bishop}
\begin{document}

\maketitle
\begin{abstract}
In this paper we compute the integral Chow ring of the moduli space of stable elliptic curves
with three marked points by combining several patching techniques, including higher Chow groups
with $\ell$-adic coefficients.
\end{abstract}

\section{Introduction}

One of the central issues of the computation of Chow rings is solving the \textit{patching problem}:
for a closed substack $Z\rightarrow X$ with complement $U$, how can one compute the Chow ring
of $X$ given the Chow rings of $U$ and $Z$? These rings fall into an exact sequence
$$
\CH(Z)\rightarrow\CH(X)\rightarrow\CH(U)\rightarrow0,
$$
but unfortunately the failure of left exactness, among many other things, makes this computation
highly non-trivial.

In \cite{Bl86}, Bloch defined \textit{higher Chow groups}, which allow one to extend this into
a long exact sequence
\begin{align*}
\dots\rightarrow\CH(Z,1)\rightarrow\CH(X,1)\rightarrow\CH(U,1)
\rightarrow\\
\CH(Z)\rightarrow\CH(X)\rightarrow\CH(U)\rightarrow0.
\end{align*}
Unfortunately still, these higher Chow groups are extremely difficult to compute and have
complicated behavior. For
instance, $\CH^1(\Spec \mathbbm k,1)=\mathbbm k^*$.
If instead we consider \textit{higher Chow groups with $\ell$-adic coefficients}, as defined in
\cite{Lar21}, we have $\CH(\Spec\mathbbm k,1;\mathbb Z_{\ell})=0$ for $\ell$ coprime to
$\fieldchar\mathbbm k$. This allows one to make computations for the non-torsion
and $\ell$-torsion parts of Chow rings.

One other philosophy is to consider enlargements of the desired stack.
As observed in \cite[the \textit{patching lemma}]{DLV21},
if the top Chern class of the normal bundle of $Z$ in $X$ is not
a zero-divisor, then the Chow ring of $X$ is exhibited as a fiber product of rings.
However, in classical intersection theory \textit{every} (degree greater than zero)
element of a Chow ring is always
a zero-divisor, since Chow rings terminate in degree equal to the dimension of the scheme.
The Chow rings of Deligne-Mumford stacks are torsion in degree higher than the dimension (since
$\CH(\mathcal X)_{\mathbb Q}\cong\CH(X)_{\mathbb Q}$ where $X$ is a coarse moduli space), and so all
(degree greater than zero) elements
are zero-divisors. Hence to satisfy this property one must enlarge the original stack to possess some
undesirable properties \textit{as a stack} so that its Chow ring can possess desirable properties.

This philosophy has been put to use in a handful of papers (\cite{Bis23, DLPV21, Per23IV})
by enlarging
the moduli stack of stable curves to include worse singularities ($A_r$ singularities).
This creates non-separated stacks whose Chow rings are better behaved, where one can
perform patching for the larger stacks and then excise out the undesired curves afterwards.

In this paper we combine the methods of higher Chow groups with $\ell$-adic coefficients with
the philosophy of enlarging stacks. In particular, we allow $A_2$ (cuspidal) singularities
\textit{and} we allow, in select circumstances, marked points to lie in the singular locus.
By carefully picking this particular modification of $\mgbar{1,3}$, we ensure that
at each step of the patching process we have that each first higher Chow group with $\ell$-adic
coefficients is zero.
We also use a variation of the patching lemma which allows us to compute Chow
rings even when the top Chern class of the normal bundle of $Z$ is a zero-divisor.
From this we get

\begin{theorem*}\label{m13bar Chow ring}
The Chow ring of $\mgbar{1,3}$ is
generated by $\lambda_1,\delta_{12},\delta_{13},\delta_{23},\delta_3$,
and the ideal of relations is generated by the following elements:
$$
\begin{aligned}
&24\lambda_1^2,\\
&\delta_{12}(\delta_{12}+\delta_3+\lambda_1),\quad
\delta_{13}(\delta_{13}+\delta_3+\lambda_1),\quad
\delta_{23}(\delta_{23}+\delta_3+\lambda_1),\\
&\delta_3(\delta_3+\delta_{12}+\lambda_1),\\
&\delta_3(\delta_{12}-\delta_{13}),\quad
\delta_3(\delta_{12}-\delta_{23}),\\
&\delta_{12}\delta_{13},\quad
\delta_{12}\delta_{23},\quad
\delta_{13}\delta_{23},\\
&12\lambda_1^2(\lambda_1+\delta_{12}+\delta_{13}-\delta_{23}+\delta_3),
\end{aligned}
$$
and
$$
\CH(\mgbar{1,3},1;\mathbb Z_{\ell})=0.
$$
\end{theorem*}

\subsection{Assumptions and notation}
We work over a fixed field $\mathbbm k$ of characteristic not equal to 2 or 3. We will always
assume that $\ell$ is coprime to $\fieldchar\mathbbm k$. When we use subscripts on $\mathbb A^n$
we are sometimes indicating the coordinates and sometimes the weight of a $\mathbb G_m$ action --
e.g. if $\mathbb G_m$ acts on $x,y$ with weights $-2,-3$, respectively, we might write
$\mathbb A^2_{x,y}$ or $\mathbb A^2_{-2,-3}$ depending on which aspect we are emphasizing.

\subsection{Acknowledgements}
I would like to thank Jarod Alper, Max Lieblich, Andrea Di Lorenzo (who, along with
Luca Battistella, independently computed the Chow ring of $\mgbar{1,3}$ and $\mgbar{1,4}$ in
\cite{BDL24}), Will Newman, Yuchen Liu, and Felix
Janda for their helpful conversations while preparing this work.

\section{Patching techniques}\label{patching techniques}
We will make use of two different patching methods: an alternate version of the
\textit{patching lemma} \cite[Lemma 3.4]{DLV21} and higher Chow groups with
$\ell$-adic coefficients \cite{Lar21}.

\begin{definition}
Define the $n^{\text{th}}$ higher Chow group with $\ell$-adic coefficients to be
$$
\CH(X,n;\mathbb Z_{\ell})=H_n\left(\lim z^*(X_{\bar{\mathbbm k}},\bullet)
\otimes^L\mathbb Z/\ell^m\mathbb Z\right).
$$
\end{definition}
In the case where each $\CH(X,n;\mathbb Z/\ell^m\mathbb Z):=H_n(z^*(X_{\bar{\mathbbm k}},\bullet)\otimes\mathbb Z/\ell^m\mathbb Z)$
is finitely generated, we have
$$
\CH(X,n;\mathbb Z_{\ell})=\lim\CH(X,n;\mathbb Z/\ell^m\mathbb Z).
$$

\begin{proposition}
If $Z\rightarrow X$ is closed with complement $U$ and:
\begin{itemize}
\item $\CH(Z)$ and $\CH(U)$ are finitely generated,
\item $\CH(Z)\rightarrow\CH(Z_{\bar{\mathbbm k}})$ is injective,
\item there exists at least one $\ell$ for which $\CH(U,1;\mathbb Z_{\ell})=0$,
\item and $\CH(U,1;\mathbb Z_l)=0$ whenever $\CH(Z)$ has $\ell$-torsion,
\end{itemize}
then the excision sequence is exact on the left.
\end{proposition}

\begin{proposition}
Suppose $\ell$ is coprime to $\fieldchar\mathbbm k$ (later we will always have $\ell=2$ or $3$).
Then
\begin{enumerate}[label=(\alph*)]
\item $\CH(\Spec \mathbbm k,1;\mathbb Z_l)=0$;
\item $\CH(\mathbb A^n,1;\mathbb Z_l)=0$;
\item $\CH(\mathbb P^n,1;\mathbb Z_l)=0$;
\item $\CH(B\mathbb G_m,1;\mathbb Z_l)=0$;
\item $\CH(B\mu_n,1;\mathbb Z_l)=0$.
\end{enumerate}
\end{proposition}
For proofs of these statements, see \cite{Bis23}. We will also make extensive use of the following
lemma, which will allow us to track the first higher Chow groups with $\ell$-adic coefficients of the various
stacks which appear during our stratification process.

\begin{lemma}\label{higher Chow patching}
Suppose that $Z\rightarrow X$ is a closed immersion with complement $U$. Suppose further that
the Chow ring of each is finitely generated and $\CH(Z)\rightarrow\CH(Z_{\bar{\mathbbm k}})$ is
injective.
Then:
\begin{enumerate}
\item if $\CH(Z,1;\mathbb Z_{\ell})$ and $\CH(U,1;\mathbb Z_{\ell})$ are both trivial, so is $\CH(X,1;\mathbb Z_{\ell})$
\item if $\CH(Z)\rightarrow\CH(X)$ is injective and $\CH(X,1;\mathbb Z_{\ell})$ is trivial, so is
$\CH(U,1;\mathbb Z_{\ell})$.
\end{enumerate}
\end{lemma}

\begin{proof}
This follows immediately from excision.
\end{proof}

Lemma 3.4 of \cite{DLV21} gives a way to express the Chow ring of $X$ as a fiber product
of the Chow rings of $U$ and $Z$ given that the top Chern class of the normal bundle of $Z$ isn't
a zero-divisor in $\CH(Z)$. This has led to the advent of several techniques where one enlarges a stack
so that the top Chern class isn't a zero-divisor. However, we will later find ourselves in a situation
where this class is a zero-divisor, and so we present a modified version of this lemma which
allows for zero-divisors.

\begin{lemma}[Porism of Lemma 3.4 \cite{DLV21}]\label{patching lemma}
Let $X$ be a smooth variety endowed with the action of a group $G$,
and let $i:Y\hookrightarrow X$ be a smooth, closed, and $G$-invariant
subvariety, with normal bundle $\mathcal N$. Consider the following cartesian diagram of rings:
$$
\begin{tikzcd}
R\arrow[r, "i^*"]\arrow[d, "j^*"] & \CH_G(Y)\arrow[d, "q"]\\
\CH_G(X\setminus Y)\arrow[r, "p"] & \CH_G(Y)/(c^G_{\text{top}}(\mathcal N))
\end{tikzcd}
$$
where the bottom horizontal arrow $p$ sends the class of a variety $V$ to the
equivalence class $i^*\eta$ where $\eta$ is any element in the set
$(j^*)\inv([V])$.

Then there is an isomorphism
$$
\frac{\CH_G(X)}{i_*(\Ann(c_{\text{top}}^G(\mathcal N)))}\xrightarrow{\sim} R.
$$
\end{lemma}

\begin{proof}
We reduce to the case where $G=\{\id\}$ and refer the reader to \cite{DLV21} for the proof
that $p$ is well-defined.

Let $V$ be a closed subvariety of $X$. Then by definition of $p$ we have $p(j^*[V])=q(i^*[V])$, so we define
a map $\varphi:\CH(X)\rightarrow R$ on cycles by $[V]\mapsto(j^*[V], i^*[V])$. Chasing elements
in the diagram shows the surjectivity of $\varphi$. Now suppose that $\varphi(\alpha)=0$.
Then we must have $j^*(\alpha)=0$, and so $\alpha=i_*(\beta)$. But then
$$
0=i^*(\alpha)=i^*i_*(\beta)=c_{\text{top}}(\mathcal N)\beta,
$$
and so $\beta$ must be in the annihilator of $c_{\text{top}}(\mathcal N)$, that is,
$\alpha\in i_*(\Ann(c_{\text{top}}(N))$.
\end{proof}

This lemma allows us, in all circumstances, to compute the Chow ring of $X$ up to elements
of $i_*(\Ann(c_{\text{top}}^G(\mathcal N)))$. In particular, it becomes more useful as the codimension
of $Z$ increases, as we will see later.

\begin{lemma}\label{ideal pushforward}
Suppose $p:Z\rightarrow X$ is closed with complement $U$. If $p^*:\CH(X)\rightarrow\CH(Z)$
is surjective, then the image of $p^*$ in $\CH(X)$ is the ideal generated by $[Z]$.
\end{lemma}

\begin{proof}
Let $\alpha\in\CH(Z)$. Then there exists $\beta\in\CH(X)$ with $p^*(\beta)=\alpha$. Then
$p_*(\alpha)=p_*p^*(\beta)=\beta p_*(1)\in([Z])$.
\end{proof}

\section{Defining a stratification}
\begin{definition}
A proper, reduced, and connected $n$-pointed curve $C$ over an algebraically closed field $K$ is said to be
\textit{$A_r$-stable} if:
\begin{enumerate}
\item $C$ has at worst $A_r$ singularities, that is, each closed point $p\in C$ has
$$
\widehat{\mathcal O}_{C,p}\cong\frac{K[[x,y]]}{(y^2-x^{h+1})}
$$
for $0\leq h\leq r$,
\item the $p_i$ are distinct and lie in the smooth locus of $C$, and
\item $\omega_C(p_1+\dots+p_n)$ is ample.
\end{enumerate}
\end{definition}

\begin{definition}
A morphism $\mathcal C\rightarrow S$ with $n$ sections $p_i:S\rightarrow\mathcal C$
is a \textit{family of $n$-pointed $A_r$-stable genus $g$ curves} if $\mathcal C\rightarrow S$
is proper, flat, and finitely presented and each geometric fiber is an $n$-pointed $A_r$-stable genus
$g$ curve.
\end{definition}

\begin{definition}
Denote by $\mgrtilde{g,n}{r}$ the stack whose objects over a scheme $S$ are families of $n$-pointed
$A_r$-stable genus $g$ curves and whose morphisms are defined in the natural way.
\end{definition}
For more about $A_r$-stable curves, see \cite{Per23}.

\begin{definition}
Define a \textit{banana curve} to be an $n$-pointed $A_1$-stable elliptic curve consisting
of only rational components with no self-nodes.
\end{definition}

\begin{figure}[H]
\begin{center}
\includegraphics[scale=.15]{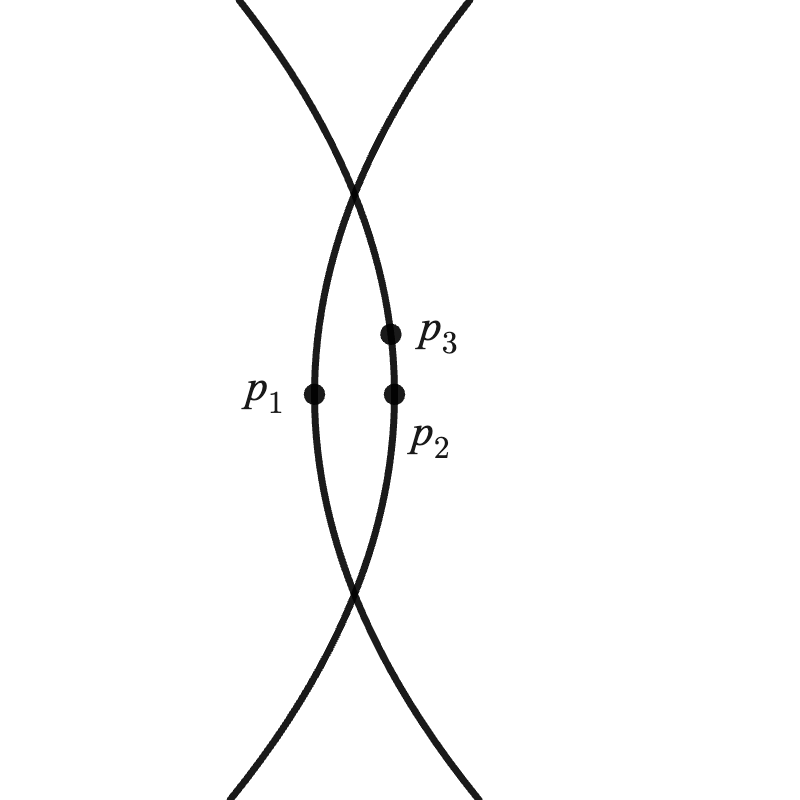}
\end{center}
\caption*{An example of a three-pointed banana curve. The lowest $n$ for which
$\mgbar{1,n}$ has a banana curve is $n=2$.}
\end{figure}

We will work with an enlargement of
$\mgrtilde{1,3}{2}\setminus\{\text{banana curves}\}$. More specifically, we allow the second marked point
to overlap with the node/cusp of a nodal/cuspidal rational curve, but we do not allow both the second
and third marked points to overlap with the node/cusp.
\begin{definition}
Let $\mathcal X$ be the stack whose objects over a scheme $S$ are proper,
flat, and finitely presented morphisms $\mathcal C\rightarrow S$
with three sections $p_i:S\rightarrow\mathcal C$
where the geometric fibers over each $s\rightarrow S$ satisfy:
\begin{itemize}
\item $(\mathcal C_{\bar s}, p_1,p_2,p_3)$ is an $A_2$-stable curve which is not a banana curve, or
\item $(\mathcal C_{\bar s},p_1)$ is an $A_2$-stable curve and we do not have
$p_1=p_2=p_3$ or $p_2=p_3=p$, where
$p$ is the self-node or cusp of a nodal or cuspidal rational curve.
\end{itemize}
\end{definition}

Using notation from the following definition, $\mathcal X$ is naturally an enlargement of
$\mgrtilde{1,3}{2}\setminus\mathcal X_{\alpha.1}$,
and loci (1)-(6) give a stratification of $\mathcal X$.

\begin{definition}\label{strata definitions}
We define the following strata, with (1)-(6) contained in $\mathcal X$ and (7) contained in $\mgbar{1,3}$.
\begin{enumerate}
\item Let $\mathcal U$ be the locus of irreducible curves where the second and third marked points
are not involutions of each other and neither point is fixed by the involution.

\item Let $\mathcal X_{\iota}$ be the locus of irreducible curves where the second and third marked
\textit{are} involutions or each other and neither point is fixed by the involution.

\item Let $\mathcal X_{23}$ be the locus of curves with one separating node
where the second and third marked points
are on the same rational component, but the first marked point is not.

\item Let $\mathcal X_{12}$ be the locus of curves with one separating node
where the first and second marked points are on the
same rational component, but the third marked point is not.

\item Let $\mathcal X_{13}$ be the locus of curves with one separating node
where the first and third marked points are on the
same rational component, but the second marked point is not.

\item Let $\mathcal X_{3}$ be the locus of curves which consist of an elliptic curve with a
three-pointed genus zero curve attached, (i.e. the
image of the gluing map from $\mgrtilde{1,1}{2}\times\mgbar{0,4}$).

\item Let $\mathcal X_{\alpha.1}$ be the locus of banana curves
where the first marked point is on its own rational component.
\end{enumerate}
\end{definition}

\begin{definition}\label{locus names}
We denote the fundamental class of $\mathcal X_*$ by $\delta_*$.
\end{definition}

Letting $\mathcal X^{\text{cusp}}$ be the locus of cuspidal curves in $\mathcal X$, we see that by
construction
$$
\mgbar{1,3}=(\mathcal X\setminus\mathcal X^{\text{cusp}})\cup\mathcal X_{\alpha.1}.
$$

We now compute the Chow rings and higher Chow groups with $\ell$-adic
coefficients of each stratum.
Our primary tool will be the Weierstrass equation for an elliptic curve, which we recap here.

\begin{theorem}[Weierstrass form for elliptic curves]\label{Weierstrass}
Any one-pointed smooth elliptic curve over a field $\mathbbm k$ of characteristic not equal to 2 or 3
can be written in the form $y^2z=x^3+axz^2+bz^3$,
where the marked point is the point at infinity $[0:1:0]$. Moreover, if we denote such a curve by
$C_{(a,b)}$, then
$$
C_{(a,b)}\cong C_{(a',b')}\quad\text{if and only if}\quad(a',b')=(t^{-4}a,t^{-6}b).
$$
The isomorphism
between these curves is given by
$$
[x:y:z]\mapsto[t^{-2}x:t^{-3}y:z].
$$
An elliptic curve
is smooth if and only if $D=4a^3+27b^2\neq0$, nodal if and only if $D=0$ and $(a,b)\neq(0,0)$, and
cuspidal if and only if $(a,b)=(0,0)$. Lastly, we have
$$
H^0(\omega_C)=\left<\frac{dx}y\right>.
$$
\end{theorem}

\begin{note}\label{hodge bundle pullback}
While one can state this theorem with each weight of the $\mathbbm k^*$ action positive,
picking negative weights
makes $\mathbbm k^*$ act on $\frac{dx}y$ with weight one, and hence the pullback of the generator
$x\in\CH(B\mathbb G_m)$ to the Chow ring of loci (1)-(7) (below) is $\lambda_1$, the first
Chern class of the Hodge bundle.
\end{note}

\begin{proposition}\label{strata Chow rings}
The Chow ring and top Chern class of the normal bundle (when relevant) of each stratum is as follows:
\begin{enumerate}
\item
$\displaystyle{\CH(\mathcal U)=\frac{\mathbb Z[\lambda_1]}{(2\lambda_1)}};$
\item
$\displaystyle{\CH(\mathcal X_{\iota})=\frac{\mathbb Z[\lambda_1]}{(3\lambda_1)}};$
\item
$\displaystyle{\CH(\mathcal X_{23})=\frac{\mathbb Z[\lambda_1]}{(12\lambda_1^2)}};$
$c_1(\mathcal N)=-\lambda_1$;
\item
$\displaystyle{\CH(\mathcal X_{12})=\mathbb Z[\lambda_1}];$
$c_1(\mathcal N)=-\lambda_1$;
\item
$\displaystyle{\CH(\mathcal X_{13})=\mathbb Z[\lambda_1]};$
$c_1(\mathcal N)=-\lambda_1$;
\item
$\displaystyle{\CH(\mathcal X_{3})=\frac{\mathbb Z[\lambda_1,x]}{(x^2)}};$
$c_1(\mathcal N)=-\lambda_1-x$;
\item
$\displaystyle{\CH(\mathcal X_{\alpha.1})=\frac{\mathbb Z[\lambda_1,x]}{(2\lambda_1,x(x+\lambda_1))}}$;
$c_2(\mathcal N)=0$.
\end{enumerate}
Additionally, the first higher Chow group with $\ell$-adic coefficients of each stratum vanishes for
$\ell=2,3$.
\end{proposition}

\begin{proof}[Proof of (1)]
The defining conditions for $\mathcal U$ precisely tell us that, letting $p_2=(x_2,y_2)$ and
$p_3=(x_3,y_3)$ be the second and third marked points written in Weierstrass form, we have
$x_2\neq x_3$. We can then solve the system of equations

\begin{align*}
y_2^2&=x_2^3+ax_2+b\\
y_3^2&=x_3^3+ax_3+b\\
\intertext{for $b$ and then $a$ to see that}
b&=y_2^2-(x_2^3+ax_2)\\
&=y_3^2-(x_3^3+ax_3)\\
a&=\frac{y_2^2-x_2^3-(y_3^2-x_3^3)}{x_2-x_3}.
\end{align*}
Therefore we can allow $x_2,y_2,x_3,y_3$ to vary freely provided that $x_2\neq x_3$. That is,
$$
\mathcal U\cong
\left[\frac{\mathbb A^2_{x_2,x_3}\setminus\{\text{diag}\}\times\mathbb A^2_{y_2,y_3}}{\mathbb G_m}\right].
$$
By Theorem \ref{Weierstrass}, $\mathbb G_m$ acts on this with weight $-2$ on the $x_i$'s and $-3$ on the $y_i$'s,
and (1) follows.
\end{proof}

\begin{proof}[Proof of (2)]
As they are involutions of each other, the third marked point is determined by the second marked point.
The condition that neither point is fixed by the involution translates to $y_2\neq0$.
Therefore
$$
\mathcal X_{\iota}\cong\left[\frac{\mathbb A^1_a\times\mathbb A^1_{x_2}\times(\mathbb A^1_{y_2}\setminus 0)}{\mathbb G_m}\right],
$$
which is a vector bundle over $B\mu_3$.
\end{proof}

\begin{proof}[Proof of (3)]
An analysis similar to the one above shows that
$$
\mathcal X_{23}\subseteq\left[\frac{\mathbb A^1_a\times\mathbb A^1_{x_2}\times\mathbb A^1_{y_2}}{\mathbb G_m}\right],
$$
and is the complement of the locus where $a+3x^2=0$ and $y=0$. This closed locus is isomorphic
to $[\mathbb A^1/\mathbb G_m]$ and has fundamental class $12\lambda_1^2$ (as it is the transverse intersection
of two divisors of weight 4 and 3, respectively). Then the Chow ring follows from excision.

To see the normal bundle statement, let $T_p$ denote the tangent space at $p$. Then
$c_1(\mathcal N_{\mathcal X_{23}})=c_1(T_{p_2})=-\psi_2=-\lambda_1$.
\end{proof}

\begin{proof}[Proof of (4) and (5)]
We use the same reasoning as in (2) and (3), except now there is no limitation on where the marked point
may go. Therefore
$$
\mathcal X_{ij}\cong\left[\frac{\mathbb A^1_a\times\mathbb A^1_{x}\times\mathbb A^1_{y}}{\mathbb G_m}\right]
$$
for $ij=12,13$. Again, similar reasoning as in (3) shows that $c_1(\mathcal N)=-\lambda_1$.
\end{proof}

\begin{proof}[Proof of (6)]
The Chow statement follows since $\mathcal X_3$ is the isomorphic image of
$$
\mgrtilde{1,1}{2}\times\mgbar{0,4}\cong\left[\frac{\mathbb A^2_{-4,-6}}{\mathbb G_m}\right]\times\mathbb P^1.
$$
And since $\mathcal X_3$ is the image of the gluing
morphism $\mgrtilde{1,1}{2}\times\mgbar{0,4}$, its normal bundle is 
$$
c_1(\mathcal N)=-\lambda_1-x\in
\CH(\mgrtilde{1,1}{2}\times\mgbar{0,4})=\mathbb Z[\lambda_1,x]/(x^2).
$$
\end{proof}

\begin{proof}[Proof of (7)]
This locus is the image of the gluing morphism from $\mgbar{0,3}\times\mgbar{0,4}\cong\mathbb P^1$.
There is an action of $\mu_2$ given by exchanging the nodes which exhibits
$\mathbb P^1\rightarrow \mathcal X_{\alpha.1}$
as a $\mu_2$-torsor, and so
$\mathcal X_{\alpha.1}\cong[\mathbb P^1/\mu_2]$ where the action has weight $(0,1)$. The claim then follows
from the projective bundle formula. We also have
$$
c_2(\mathcal N)=p^*(\delta_{\alpha.1})=
p^*(6\lambda_1(\lambda_1+\delta_{12}+\delta_{13}-\delta_{23}+\delta_3))=0,
$$
by the proof of Theorem \ref{m13bar Chow ring}.
\end{proof}

\begin{proposition}
The stack $\mathcal X$ is smooth.
\end{proposition}

\begin{proof}
Consider $Y=V(y_i^2-(x_i^3+ax_i+b))_{i=2,3}\subseteq\mathbb A^6$ with coordinates $a,b,x_i,y_i$. By similar
arguments as above,
$[Y/\mathbb G_m]$ parametrizes irreducible one-pointed affine elliptic curves with two additional marked points.

Let $Y^{\circ}\subset Y$ be the complement of the locus of points which in $[Y/\mathbb G_m]$ give
curves where the second and third marked points overlap with a node/cusp.
By the Jacobian criterion $Y^{\circ}$
is smooth, and hence so is $[Y^{\circ}/\mathbb G_m]$. Therefore $\mathcal X$
is smooth as it is the union of two smooth opens:
$$
\mathcal X=\left(\mgrtilde{1,3}{2}\setminus\{\text{banana curves}\}\right)\cup[Y^{\circ}/\mathbb G_m].
$$
\end{proof}

\begin{proposition}
The stack $\mathcal X$ is a quotient stack.
\end{proposition}

\begin{proof}
This follows from the following two lemmas.
\end{proof}

\begin{lemma}\label{quotient stack lemma}
Suppose that $\mathcal Y$ is a stack which stratifies into locally closed $\mathcal Y_i$ such that each
$\mathcal Y_i \cong[Y_i/G]$ is a quotient stack (note that we can always assume any finite collection
of quotient stacks is expressible as a quotient by the same group).
Then $\mathcal Y$ is a quotient stack if $\mathcal Y$ admits a morphism to $BG$
compatible with each $\mathcal Y_i\rightarrow\mathcal Y$.
\end{lemma}

\begin{proof}
Consider the diagram
$$
\begin{tikzcd}
Y\arrow[r]\arrow[d] & \Spec\mathbbm k\arrow[d]\\
\mathcal Y\arrow[r] & BG
\end{tikzcd}
$$
Then $Y\rightarrow\mathcal Y$ is a principal $G$-bundle, and it suffices to show that
$Y$ is an algebraic space. But now the diagram
$$
\begin{tikzcd}
Y_i\arrow[r]\arrow[d]&Y\arrow[r]\arrow[d] & \Spec\mathbbm k\arrow[d]\\
\mathcal Y_i\arrow[r]&\mathcal Y\arrow[r] & BG
\end{tikzcd}
$$
shows that $Y$ stratifies into algebraic spaces and hence is an algebraic space itself.
\end{proof}

\begin{lemma}
The above stratification of $\mathcal X$ satisfies the conditions of Lemma \ref{quotient stack lemma}.
\end{lemma}

\begin{proof}
By Note \ref{hodge bundle pullback} the pullback of the universal
line bundle on $B\mathbb G_m$ to each of stacks (1)-(6) above is the Hodge bundle. Letting
$\mathcal X\rightarrow B\mathbb G_m$ be the morphism determined by the Hodge bundle, we see
that the inclusion of each stratum into $\mathcal X$ is compatible with the morphisms to $B\mathbb G_m$,
since all are given by the Hodge bundle.
\end{proof}

\section{Computing $\CH(\mgbar{1,3})$}
Abusing notation, if $Z$ is a locally closed stratum from Definition \ref{strata definitions}, we will
denote the inclusion by $p:Z\rightarrow\mathcal X$. Recall from Definition \ref{locus names}
that we write $\delta_*$ for the fundamental class of $\mathcal X_*$.
\begin{proposition}
The Chow ring of $\mathcal U\cup\mathcal X_{\iota}$ is
$$
\CH(\mathcal U\cup\mathcal X_{\iota})=\mathbb Z[\lambda_1]/(6\lambda_1^2)
$$
and the first higher Chow group with $\ell$-adic coefficients is trivial.
\end{proposition}

\begin{proof}
See \cite[Lemma 4.9]{Bis23}.
\end{proof}

\begin{proposition}\label{step 3}
The Chow ring of $\mathcal U\cup\mathcal X_{\iota}\cup\mathcal X_{23}$ is
$$
\CH(\mathcal U\cup\mathcal X_{\iota}\cup\mathcal X_{23})=
\frac{\mathbb Z[\lambda_1,\delta_{23}]}{6\lambda_1(\lambda_1-\delta_{23}), \delta_{23}(\delta_{23}+\lambda_1)},
$$
and the first higher Chow group with $\ell$-adic coefficients is zero.
\end{proposition}

\begin{proof}
We have
$$
p_*(\Ann(c_1(\mathcal N)))=p_*(\Ann(-\lambda_1))=p_*((12\lambda_1))=
p_*((12 p^*\lambda_1))=(12\lambda_1\delta_{23}),
$$
and so the following diagram is cartesian:
$$
\begin{tikzcd}
\CH(\mathcal U\cup\mathcal X_{\iota}\cup\mathcal X_{23})/(12\lambda_1\delta_{23})
\arrow[r, "p^*"]\arrow[d, "j^*"] & \mathbb Z[\lambda_1]/(12\lambda_1^2)\arrow[d]\\
\mathbb Z[\lambda_1]/(6\lambda_1^2)\arrow[r] & \mathbb Z.
\end{tikzcd}
$$
Since $\CH(\mathcal U\cup\mathcal X_{\iota}\cup\mathcal X_{23})$ surjects onto the fiber product
and is generated by $\lambda_1$ and $\delta_{23}$, to compute the fiber product it is sufficient
to see which polynomials in $\lambda_1$ and $\delta_{23}$ are in the kernel of both
$j^*$ and $p^*$.
We find the following relations on the fiber product of rings:
$$
\delta_{23}(\delta_{23}+\lambda_1)=0,\quad 6\lambda_1(\lambda_1+\delta_{23})=0,
\quad 12\lambda_1\delta_{23}=0,
$$
and so these relations hold on $\mathcal U\cup\mathcal\mathcal X_{\iota}\cup\mathcal X_{23}$
modulo $12\lambda_1\delta_{23}$.

First observe that since the first higher Chow group with $\ell$-adic coefficients of $\mathcal U\cup
\mathcal X_{\iota}$ is zero, we must have that $p_*$ is injective. Therefore, as $12\lambda_1\delta_{23}=
p_*(12\lambda_1)$, the third of the above relations does not hold. Second, observe that since
$$
\delta_{23}^2=p_*(1)p_*(1)=p_*(p^*p_*(1))=p_*(-\lambda_1)=-\lambda_1\delta_{23},
$$ we have $\delta_{23}(\delta_{23}+\lambda_1)=0$. Lastly, for the second relation,
observe that we must have
$$
6\lambda_1(\lambda_1+\delta_{23})+a\cdot12\lambda_1\delta_{23}=0
$$
for some $a$. But since $\mathcal X_{\alpha.1}$, $\mathcal X_{12}$, $\mathcal X_{13}$,
and $\mathcal X_3$ are still excised, we have (by Lemma \ref{WDVV})
$$
0=4[\mathcal X_{\alpha.1}]=4\delta_{\alpha.1}=24\lambda_1(\lambda_1+\delta_{12}+\delta_{13}-
\delta_{23}+\delta_3)=24\lambda_1(\lambda_1-\delta_{23}),
$$
which then implies that $a=-1$ and concludes the proof.
\end{proof}

The next few steps now proceed smoothly, as at each step the top Chern class of the normal bundle
is not a zero-divisor.

\begin{proposition}\label{step 4}
The Chow ring of $\mathcal U\cup\mathcal X_{\iota}\cup\mathcal X_{23}\cup\mathcal X_{12}$ is
generated by $\lambda_1,\delta_{12},\delta_{23}$, and the ideal of relations is generated by the following
elements:
$$
6\lambda_1(\lambda_1+\delta_{12}-\delta_{23}),\quad
\delta_{12}(\delta_{12}+\lambda_1),\quad
\delta_{23}(\delta_{23}+\lambda_1),\quad
\delta_{12}\delta_{23}.
$$
Moreover, the first higher Chow group with $\ell$-adic coefficients is trivial.
\end{proposition}

\begin{proof}
Since $c_1(\mathcal N)=-\lambda_1$ is not a zero-divisor on $\CH(\mathcal X_{12})=\mathbb Z[\lambda_1]$,
Lemma \ref{patching lemma} implies that
$$
\begin{tikzcd}
\CH(\mathcal U\cup\mathcal X_{\iota}\cup\mathcal X_{23}\cup\mathcal X_{12})
\arrow[r, "p^*"]\arrow[d, "j^*"] & \mathbb Z[\lambda_1]\arrow[d]\\
\CH(\mathcal U\cup\mathcal X_{\iota}\cup\mathcal X_{23})
\arrow[r] & \mathbb Z
\end{tikzcd}
$$
is cartesian. As before, the relations in the fiber product are the intersection of the kernels of $p^*$ and $j^*$.
So assume that $f(\lambda_1,\delta_{12},\delta_{23})$ is a relation. Then $j^*f=0$ implies that
$f$ is of the form
$$f=f'(\lambda_1,\delta_{23})+\delta_{12}g(\lambda_1,\delta_{12},\delta_{23})
$$
where $f'$ is a relation on $\mathcal U\cup\mathcal X_{\iota}\cup\mathcal X_{23}$. We check both of
the relations found in Proposition \ref{step 3}, applying $p^*$:
$$
0=p^*f=p^*(6\lambda_1(\lambda_1-\delta_{23})+\delta_{12}g(\lambda_1,\delta_{12},\delta_{23}))=
6\lambda_1^2-\lambda_1g(\lambda_1,-\lambda_1,0).
$$
Since $\lambda_1$ isn't a zero-divisor in $\CH(\mathcal X_{12})=\mathbb Z[\lambda_1]$, we may cancel a
$\lambda_1$ and obtain $g(\lambda_1,-\lambda_1,0)=6\lambda_1^2$, which implies that
$g=6\lambda_1$, up to an element of the kernel of $p^*$. This gives us our first relation:
$$
6\lambda_1(\lambda_1-\delta_{23})+\delta_{12}\cdot 6\lambda_1=6\lambda_1(\lambda_1+\delta_{12}-
\delta_{23})=0.
$$
Similar reasoning applied to the other relation from Proposition \ref{step 3} gives:
$$
\delta_{23}(\delta_{23}+\lambda_1)=0.
$$
We now need to check elements from the kernel of $p^*$, as if $p^*\alpha=0$, then
$\delta_{12}\alpha$ is in the kernel of both $p^*$ and $j^*$. So suppose $p^*f''(\lambda_1,\delta_{12},
\delta_{23})=0$. Then
$$
0=p^*f''=f(\lambda_1,-\lambda_1,0),
$$
and so generators for such polynomials are $f''=\delta_{23}$ and $f''=\delta_{12}+\lambda_1$,
which gives two additional relations:
$$
\delta_{12}\delta_{23}=0,\quad \delta_{12}(\delta_{12}+\lambda_1)=0.
$$
\end{proof}

\begin{proposition}\label{step 5}
The Chow ring of $\mathcal U\cup\mathcal X_{\iota}\cup\mathcal X_{23}\cup\mathcal X_{12}\cup
\mathcal X_{13}$ is
generated by $\lambda_1,\delta_{12},\delta_{13},\delta_{23}$,
and the ideal of relations is generated by the following elements:
$$
\begin{aligned}
&\delta_{12}(\delta_{12}+\lambda_1),\quad
\delta_{13}(\delta_{13}+\lambda_1),\quad
\delta_{23}(\delta_{23}+\lambda_1),\\
&\delta_{12}\delta_{13},\quad
\delta_{12}\delta_{23},\quad
\delta_{13}\delta_{23},\\
&6\lambda_1(\lambda_1+\delta_{12}+\delta_{13}-\delta_{23}),
\end{aligned}
$$
and the first higher Chow group with $\ell$-adic coefficients is trivial.
\end{proposition}

\begin{proof}
Once again we have $c_1(\mathcal N)=-\lambda_1$, which is not a zero-divisor in
$\CH(\mathcal X_{13})=\mathbb Z[\lambda_1]$. Therefore we may apply Lemma \ref{patching lemma} to
directly compute the Chow ring.
This computation follows exactly as it does in Proposition
\ref{step 4}, so we omit it.
\end{proof}

\begin{proposition}\label{step 6}
The Chow ring of $\mathcal X$ is
generated by $\lambda_1,\delta_{12},\delta_{13},\delta_{23},\delta_3$,
and the ideal of relations is generated by the following elements:
$$
\begin{aligned}
&\delta_{12}(\delta_{12}+\delta_3+\lambda_1),\quad
\delta_{13}(\delta_{13}+\delta_3+\lambda_1),\quad
\delta_{23}(\delta_{23}+\delta_3+\lambda_1),\\
&\delta_3(\delta_3+\delta_{12}+\lambda_1),\\
&\delta_3(\delta_{12}-\delta_{13}),\quad
\delta_3(\delta_{12}-\delta_{23}),\\
&\delta_{12}\delta_{13},\quad
\delta_{12}\delta_{23},\quad
\delta_{13}\delta_{23},\\
&6\lambda_1(\lambda_1+\delta_{12}+\delta_{13}-\delta_{23}+\delta_3),
\end{aligned}
$$
and the first higher Chow group with $\ell$-adic coefficients vanishes.
\end{proposition}

\begin{proof}
Since $\mathcal X_3$ is the image of the gluing
morphism $\mgrtilde{1,1}{2}\times\mgbar{0,4}$, its normal bundle is 
$$
c_1(\mathcal N)=-\lambda_1-x\in
\CH(\mgrtilde{1,1}{2}\times\mgbar{0,4})=\mathbb Z[\lambda_1,x]/(x^2),
$$
which is not a zero-divisor. Therefore we can once again use Lemma \ref{patching lemma} to directly
compute the Chow ring. We show a step of the computation here, and then leave the rest to the reader.

We first need to analyze the pullback homomorphism to $\CH(\mathcal X_3)$. One sees the following:
$$
p^*(\lambda_1)=\lambda_1,\quad p^*(\delta_{12})=p^*(\delta_{13})=p^*(\delta_{23})=x,\quad
p^*(\delta_3)=c_1(\mathcal N)=-\lambda_1-x.
$$
Therefore $p^*$ is surjective, and so by Lemma \ref{ideal pushforward}
the image of $p_*$ is the ideal generated by $\delta_3$, implying
that $\CH(\mathcal X)$ is generated by $\lambda_1,\delta_{12},\delta_{13},\delta_{23},\delta_3$.

The relations on the fiber product are then polynomials $f$ in the above variables which are in the kernel
of both $j^*$ and $p^*$. We then see:
$$
j^*(f)=0\implies f= f'(\lambda_1,\delta_{12},\delta_{13},\delta_{23})+\delta_3g(\lambda_1,
\delta_{12},\delta_{13},\delta_{23},\delta_3)
$$
where $f'$ is in the ideal of relations from Proposition \ref{step 5}. Applying $p^*$ gives
$$
0=p^*(f)=f'(\lambda_1,x,x,x)+(-\lambda_1-x)g(\lambda_1,x,x,x,-\lambda_1-x).
$$
The first relation from Proposition \ref{step 5} is the most complicated, so we demonstrate how it works in this
case. Plugging in we have
$$
0=6\lambda_1(\lambda_1+x)-(\lambda_1+x)g(\lambda_1,x,x,x,-\lambda_1-x).
$$
Since $\lambda_1+x$ isn't a zero-divisor we may cancel and obtain
$$
g(\lambda_1,x,x,x,-\lambda_1-x)=6\lambda_1
$$
and so, up to an element of the kernel of $p^*$, $g=6\lambda_1$.
This then gives the relation
$$
f=f'+\delta_3g=6\lambda_1(\lambda_1+\delta_{12}+\delta_{13}-\delta_{23}+\delta_3)=0.
$$
\end{proof}

So far the stack we have been building contains cuspidal curves, which aren't of interest to us here. So now
we must excise the cuspidal locus.

\begin{lemma}\label{cusp Chow group}
Let $\mathcal X^{\text{cusp}}$ denote the locus of cuspidal curves inside
$\mathcal X$. Then there is a smooth stack $\mathcal Y$ with a morphism
$\eta:\mathcal Y\rightarrow\mathcal X^{\text{cusp}}$ such that:
\begin{itemize}
\item the pushforward $\eta_*$ induces an isomorphism on Chow groups and higher Chow groups,
\item $\CH(\mathcal Y,1;\mathbb Z_{\ell})=0$, and
\item
the Chow ring of $\mathcal Y$ is generated by the pullbacks of
$\lambda_1,\delta_{12},\delta_{13},\delta_{23},\delta_3$ from $\mathcal X$, i.e. the pullback is surjective.
\end{itemize}
\end{lemma}

\begin{proof}
The stack $\mathcal X^{\text{cusp}}$ parametrizes cuspidal curves with
three marked points where the first marked point is the point at infinity, and so it inherits its singularities from
curves where the markings lie on the cusp. If we smooth these cuspidal curves into $\mathbb P^1$'s and
mark the preimage of the cusp as a fourth marked point, the resulting
stack, which we call $\mathcal Y$, parametrizes four points on $\mathbb P^1$ subject to the following conditions:
\begin{itemize}
\item the first and fourth marked point are $\infty$ and 0, respectively
\item the second and third marked points may not both overlap with 0.
\end{itemize}
Note that this is \textit{not} directly related to $\mgbar{0,4}$: for instance, we allow the following curves

\begin{figure}[H]
\begin{center}
\includegraphics[scale=.15]{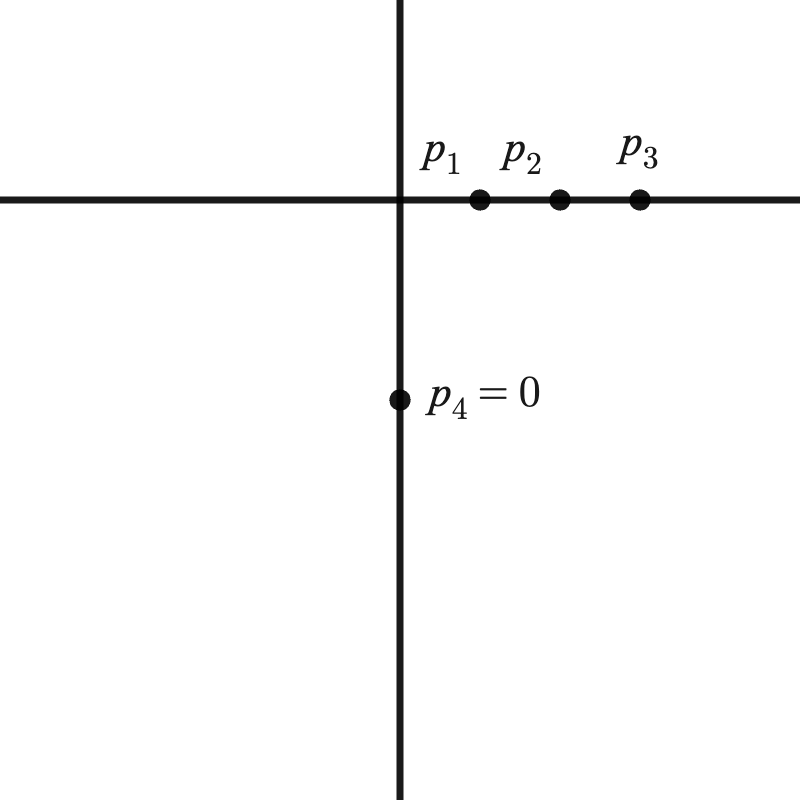}
\end{center}
\caption*{A curve in $\mathcal Y$ which is not in $\mgbar{0,4}$. These curves
are the isomorphic image of $\mg{0,2}\times\mgbar{0,4}$.}
\end{figure}

The morphism $\eta:\mathcal Y\rightarrow\mathcal X^{\text{cusp}}$
given by pinching the fourth marked point is readily
seen to be a Chow envelope and hence induces a surjection on Chow groups. While Chow envelopes
do not in general induce surjections on higher Chow groups, an application of the excision sequence
(using the below stratification) shows
that in this case $\eta_*$ is an isomorphism on higher Chow.
Therefore it is enough to compute using $\mathcal Y$ instead by considering the composition
$$
\begin{tikzcd}
\CH(\mathcal Y)\arrow[rr, bend left]\arrow[r] & \CH(\mathcal X^{\text{cusp}})\arrow[r] & \CH(\mathcal X).
\end{tikzcd}
$$

There is an analogous stratification of $\mathcal Y$ to the stratification of $\mathcal X$,
as follows.
\begin{enumerate}
\item Curves where $p_2,p_3\neq\infty$ and we do not have $p_2=p_3=0$. This locus is
isomorphic to $\mathbb P^1$. The generator of $\CH(\mathbb P^1)$ corresponds to the pullback
of both $\lambda_1$ and $\delta_{23}$.
\item Curves where $p_2$ overlaps with $\infty$ but $p_3$ doesn't. This locus is isomorphic
to $[\mathbb A^1/\mathbb G_m]$ and geometrically corresponds to the pullback of $\delta_{12}$.
\item Curves where $p_3$ overlaps with $\infty$ but $p_2$ doesn't. This locus is isomorphic
to $[\mathbb A^1/\mathbb G_m]$ and geometrically corresponds to the pullback of $\delta_{13}$.
\item Curves shown in the figure above. This locus is isomorphic to $\mg{0,2}\times\mgbar{0,4}\cong
B\mathbb G_m\times\mathbb P^1$ and geometrically corresponds to the pullback of $\delta_3$.
\end{enumerate}
The first higher Chow group with $\ell$-adic coefficients of each of these loci is trivial, and so
we conclude that $\CH(\mathcal Y,1;\mathbb Z_{\ell})=0$.

We also stratify $\mathcal Y$ as follows:
$$
\mathcal Y=\left[\frac{(\mathbb P^1\times\mathbb P^1)\setminus\{(\infty,\infty), (0,0)\}}{\mathbb G_m}\right]
\sqcup(B\mathbb G_m\times\mathbb P^1).
$$
With the geometric interpretation of loci (1)-(4) above, we see, by the projective bundle formula,
excision, and patching, that the Chow ring of $\mathcal Y$ is generated by
the pullback of $\lambda_1,\delta_{12},\delta_{13},\delta_{23},\delta_3$ from
$\mathcal X$. We omit the actual computation of $\CH(\mathcal Y)$ since this is all we needed.
\end{proof}

\begin{note}
The Chow ring of $\mathcal Y$ can be directly computed, but all that matters here is its generators.
The Chow group and first higher Chow group of $\mathcal X^{\text{cusp}}$
can also be computed via patching, with no need for Chow envelopes -- however,
one then runs into issues with the lack of a ring structure on $\CH(\mathcal X^{\text{cusp}})$.
\end{note}

\begin{lemma}\label{cusp fundamental class}
The fundamental class of the cuspidal locus is $24\lambda_1^2$.
\end{lemma}

\begin{proof}
This follows by pulling back along the morphism
$\mathcal X\rightarrow\mgrtilde{1,1}{2}$ which forgets $p_2$ and $p_3$,
since in
$$
\mgrtilde{1,1}{2}\cong\left[\frac{\mathbb A_{-4,-6}}{\mathbb G_m}\right]
$$
the fundamental class of the cuspidal locus is $[(0,0)]=24\lambda_1^2$.
\end{proof}

\begin{proposition}
The Chow ring of $\mgbar{1,3}\setminus\mathcal X_{\alpha.1}$ is
generated by $\lambda_1,\delta_{12},\delta_{13},\delta_{23},\delta_3$,
and the ideal of relations is generated by the following elements:
$$
\begin{aligned}
&24\lambda_1^2,\\
&\delta_{12}(\delta_{12}+\delta_3+\lambda_1),\quad
\delta_{13}(\delta_{13}+\delta_3+\lambda_1),\quad
\delta_{23}(\delta_{23}+\delta_3+\lambda_1),\\
&\delta_3(\delta_3+\delta_{12}+\lambda_1),\\
&\delta_3(\delta_{12}-\delta_{13}),\quad
\delta_3(\delta_{12}-\delta_{23}),\\
&\delta_{12}\delta_{13},\quad
\delta_{12}\delta_{23},\quad
\delta_{13}\delta_{23},\\
&6\lambda_1(\lambda_1+\delta_{12}+\delta_{13}-\delta_{23}+\delta_3),
\end{aligned}
$$
and
$$
\CH(\mgbar{1,3}\setminus\mathcal X_{\alpha.1},1;\mathbb Z_{\ell})=0.
$$
\end{proposition}

\begin{proof}
By Lemma \ref{cusp Chow group} the pullback
$$
\CH(\mathcal X)\rightarrow\CH(\mathcal Y)
$$
is surjective, and by Lemma \ref{cusp fundamental class} we have $[\mathcal Y]=24\lambda_1^2$. Therefore
by Lemma \ref{ideal pushforward},
we get that the image of $p_*$ is the ideal generated by $24\lambda_1^2$,
and so the Chow ring of
$\mgbar{1,3}\setminus\mathcal X_{\alpha.1}=\mathcal X\setminus\mathcal X^{\text{cusp}}$
has the same generators as in the previous step, but with the
added relation $24\lambda_1^2=0$. In fact, we see that the pushforward map is multiplication by
$24\lambda_1^2$ and hence is injective. Therefore by Lemma \ref{higher Chow patching}
we conclude $\CH(\mgbar{1,3}\setminus\mathcal X_{\alpha.1},1;\mathbb Z_{\ell})=0$.
\end{proof}

Before we can finish the computation, we need a quick analysis of the fundamental class, $\delta_{\alpha.1}$,
of the next locus, $\mathcal X_{\alpha.1}$, following \cite[Lemma 3.2.1]{Bel98}.

\begin{lemma}\label{WDVV}
The following identity holds in $\CH(\mgbar{1,3})$:
$$
4\delta_{\alpha.1}=24\lambda_1(\delta_{12}+\delta_{13}-\delta_{23}+\delta_3).
$$
\end{lemma}

\begin{proof}
There is a map $f:\mgbar{0,5}\rightarrow
\mgbar{1,3}$ given by identifying the fourth and fifth marked points. Using the notation of
\cite{Kee92}, we let $D_{ij}$ denote the divisor on $\mgbar{0,5}$ where the $i^{\text{th}}$ and $j^{\text{th}}$
marked points coincide. Then the WDVV relations give
$$
\begin{aligned}
D_{12}+D_{45}&=D_{14}+D_{25}\\
D_{13}+D_{45}&=D_{14}+D_{35}\\
D_{23}+D_{45}&=D_{24}+D_{35}.
\end{aligned}
$$ Pushing these
relations forward to $\mgbar{1,3}$ gives
$$
\begin{aligned}
2(12\lambda_1\delta_{12}+12\lambda_1\delta_3)&=
2(\delta_{\alpha.1}+\delta_{\alpha.2})\\
2(12\lambda_1\delta_{13}+12\lambda_1\delta_3)&=
2(\delta_{\alpha.1}+\delta_{\alpha.3})\\
2(12\lambda_1\delta_{23}+12\lambda_1\delta_3)&=
2(\delta_{\alpha.2}+\delta_{\alpha.3})
\end{aligned}
$$ where $\delta_{\alpha.i}$ is the locus of banana curves where the
$i^{\text{th}}$ point is on its own rational component, and where the extra factor of two comes the fact that
on each divisor the map $f$ is degree two. Solving this system of equations for $\delta_{\alpha.1}$
gives the desired formula.
\end{proof}

\begin{theorem}\label{m13bar Chow ring}
The Chow ring of $\mgbar{1,3}$ is
generated by $\lambda_1,\delta_{12},\delta_{13},\delta_{23},\delta_3$,
and the ideal of relations is generated by the following elements:
$$
\begin{aligned}
&24\lambda_1^2,\\
&\delta_{12}(\delta_{12}+\delta_3+\lambda_1),\quad
\delta_{13}(\delta_{13}+\delta_3+\lambda_1),\quad
\delta_{23}(\delta_{23}+\delta_3+\lambda_1),\\
&\delta_3(\delta_3+\delta_{12}+\lambda_1),\\
&\delta_3(\delta_{12}-\delta_{13}),\quad
\delta_3(\delta_{12}-\delta_{23}),\\
&\delta_{12}\delta_{13},\quad
\delta_{12}\delta_{23},\quad
\delta_{13}\delta_{23},\\
&12\lambda_1^2(\lambda_1+\delta_{12}+\delta_{13}-\delta_{23}+\delta_3),
\end{aligned}
$$
and
$$
\CH(\mgbar{1,3},1;\mathbb Z_{\ell})=0.
$$
\end{theorem}

\begin{proof}
All that's left to do is to patch in $\mathcal X_{\alpha.1}$, whose Chow ring
is $\mathbb Z[\lambda_1,x]/(2\lambda_1, \lambda_1(\lambda_1+x))$, by Proposition \ref{strata Chow rings}.
The pullback map is given by
$$
\begin{aligned}
&\lambda_1\mapsto\lambda_1; &\delta_{12}, \delta_{13}, \delta_3\mapsto 0;\ \ 
&\delta_{23}\mapsto x; &\delta_{\alpha.1}\mapsto c_2(\mathcal N),
\end{aligned}
$$
and so the image of the pushforward is the ideal generated by $\delta_{\alpha.1}$, by Lemma
\ref{ideal pushforward}. By Lemma \ref{WDVV} and the injectivity
of the pushforward,
we have
$$
4\delta_{\alpha.1}=24(\delta_{12}+
\delta_{13}-\delta_{23}+\delta_3)=4\cdot 6\lambda_1(\lambda_1+\delta_{12}+\delta_{13}-\delta_{23}+\delta_3)
\neq0.
$$
Hence $6\lambda_1(\lambda_1+\delta_{12}+\delta_{13}-\delta_{23}+\delta_3)$ is in the image of
$p_*$, and so is equal to $p_*(a)$ for some $a$. But then
$$
\begin{aligned}
p_*(4a)&=24\lambda_1(\lambda_1+\delta_{12}+\delta_{13}-\delta_{23}+\delta_3)\\
a\cdot p_*(4)&=a\cdot24\lambda_1(\lambda_1+\delta_{12}+\delta_{13}-\delta_{23}+\delta_3)\\
0&=(a-1)\cdot24\lambda_1(\lambda_1+\delta_{12}+\delta_{13}-\delta_{23}+\delta_3)\\
&=p_*(4(a-1)),
\end{aligned}
$$
and we conclude that $a=1$, again by the injectivity of $p_*$. Therefore the image of $p_*$ is the
ideal generated by $6\lambda_1(\lambda_1+\delta_{12}+\delta_{13}-\delta_{23}+\delta_3)$, and so
there are no new generators added from the previous patching step.
Thus we only need to verify the claimed
relations. The first ten of the relations held on the previous locus and hence will hold on $\mgbar{1,3}$ up to
an element of $(\delta_{\alpha.1})$. We give an example of how to quickly verify them.

Let $\sigma_{13}:\mgbar{1,2}\rightarrow\cgbar{1,2}\cong\mgbar{1,3}$ be the section which sets $p_3=p_1$.
Then we have
$$
\delta_{13}(\delta_{13}+\delta_3+\lambda_1)=\sigma_{13,*}\sigma_{13}^*(\delta_{13}+\delta_3+\lambda_1)=
\sigma_{13,*}(-\lambda_1-\mu_1+\mu_1+\lambda_1)=0,
$$
using the notation of \cite{DLPV21}. The others follow similarly.

For the eleventh and final relation, observe that we must have
$$
0=p_*(0)=p_*(2\lambda_1)=2\lambda_1\cdot 6\lambda_1(\lambda_1+\delta_{12}+\delta_{13}-\delta_{23}+
\delta_3),
$$
which gives the eleventh relation and concludes the proof.
\end{proof}

\begin{note}
Tensoring with $\mathbb Q$, we see that Theorem \ref{m13bar Chow ring} agrees
with Theorem 3.3.2 from \cite{Bel98}.
\end{note}

\begin{note}
One alternative perspective on verifying the relations is this. We \textit{a priori} know the
first ten relations have to hold
(as they come from pushing and pulling along various tautological maps from spaces whose
Chow rings are already known),
and we can view the crux of this proof as obtaining the eleventh relation, which we get from
understanding the correct stack structure on $\mathcal X_{\alpha.1}$ (along with the vanishing
of the first higher Chow groups with $\ell$-adic coefficients).
\end{note}

\newpage
\bibliographystyle{plain}
\bibliography{/Users/martinbishop/documents/research/citations/bibliography.bib}

\end{document}